\documentclass{amsart}
\usepackage{color}
\usepackage{amsthm}
\usepackage{latexsym}
\usepackage{dsfont}
\usepackage{bbm}
\usepackage{amssymb}
\usepackage{amsmath}
\usepackage{graphicx}

\numberwithin{equation}{section}

\theoremstyle{plain}
\newtheorem{Theorem}{Theorem}[section]
\newtheorem{Lemma}[Theorem]{Lemma}

\theoremstyle{remark}
\newtheorem{Rem}[Theorem]{Remark}
\theoremstyle{definition}

\DeclareMathOperator{\R}{\mathbb{R}}

\DeclareMathOperator{\Prob}{\mathbb{P}}

\DeclareMathOperator{\E}{\mathbb{E}}

\DeclareMathOperator{\1}{\mathbbm{1}}

\newcommand{\tofd}{\overset{{\rm f.d.d.}}{\underset{n\to\infty}\longrightarrow}}
\newcommand{\tofdt}{\overset{{\rm f.d.d.}}{\underset{t\to\infty}\longrightarrow}}

\newcommand{\todistr}{\overset{{\rm d}}{\underset{n\to\infty}\longrightarrow}}
\newcommand{\todistrt}{\overset{{\rm d}}{\underset{t\to\infty}\longrightarrow}}

\newcommand{\toweakt}{\underset{t\to\infty}{\Rightarrow}}

\newcommand{\ton}{\overset{}{\underset{n\to\infty}\longrightarrow}}

\DeclareMathOperator{\N}{\mathbb{N}}

\newcommand{\mm}{\mathcal{Z}}
\newcommand{\mn}{\N}

\newcommand{\mr}{\mathbb{R}}

\newcommand{\eps}{\varepsilon}

\begin{document}

\title[Intermediate levels of random recursive trees]{Weak convergence of the number of vertices at intermediate levels
of random recursive trees}
\author{Alexander Iksanov}
\address{Faculty of Computer Science and Cybernetics, Taras Shevchenko National University of Kyiv, 01601 Kyiv, Ukraine}
\email{iksan@univ.kiev.ua}
\author{Zakhar Kabluchko}
\address{Institut f\"{u}r Mathematische Statistik, Westf\"{a}lische Wilhelms-Universit\"{a}t M\"{u}nster,
48149 M\"{u}nster, Germany}
\email{zakhar.kabluchko@uni-muenster.de}

\begin{abstract}
Let $X_n(k)$ be the number of vertices at level $k$ in a random
recursive tree with $n+1$ vertices. We are interested in the
asymptotic behavior of $X_n(k)$ for intermediate levels $k=k_n$
satisfying $k_n\to\infty$ and $k_n=o(\log n)$ as $n\to\infty$. In
particular, we prove weak convergence of finite-dimensional
distributions for the process $(X_n ([k_n u]))_{u>0}$, properly
normalized and centered, as $n\to\infty$. The limit is a centered
Gaussian process with covariance $(u,v)\mapsto (u+v)^{-1}$.
One-dimensional distributional convergence of $X_n(k_n)$, properly
normalized and centered, was obtained with the help of analytic
tools by Fuchs, Hwang and Neininger in
\cite{Fuchs+Hwang+Neininger:2006}. In contrast, our proofs which
are probabilistic in nature exploit a connection of our model with
certain Crump-Mode-Jagers branching processes.

\end{abstract}
\maketitle

%2000 Mathematics Subject Classification: Primary: 60J80 \\
%\hphantom{2000 Mathematics Subject Classification: }Secondary: 60K05, 60G42

\section{Introduction and main result}
A (deterministic) {\it recursive tree} with $n$ vertices is a
rooted tree with vertices labeled with $1,2\ldots, n$ that
has the following property: the labels of the vertices on the unique path from the root (labeled with $1$) to
any other vertex (labeled with $m\in\{2,\ldots, n\}$) form an
increasing sequence. There are $(n-1)!$ different recursive trees
with $n$ vertices, and we denote them $T_{1,n}, T_{2,n},\ldots,
T_{(n-1)!, n}$. A random object $\mathcal{T}_n$ is called {\it
random recursive tree} with $n$ vertices if it has uniform distribution on the set of recursive trees with $n$ vertices, that is,
$$\Prob\{\mathcal{T}_n=T_{i,n}\}=\frac{1}{(n-1)!},\quad i=1,2,\ldots, (n-1)!.$$

Let $X_n(k)$ be the number of vertices at level
$k\in \mn$ (that is, at distance $k$ from the root) in the random recursive
tree $\mathcal{T}_{n+1}$ on $n+1$ vertices.
It is known that $\mathcal{T}_{n+1}$ has
logarithmic height (see Theorem 1 in \cite{Pittel:1994} and~\cite{devroye:1987}), namely
$$
\frac {\max\{k\in\N\colon X_n(k) \neq 0 \}}{\log n} \ton e \quad \text{a.s.}
$$
The asymptotic behavior of the occupation numbers $X_n(k)$ as $n\to\infty$ has been much studied for various asymptotic regimes of $k=k_n$ that is allowed to be a function of $n$. In Theorem 3 of \cite{Fuchs+Hwang+Neininger:2006} it was shown by using analytic tools that for any fixed $k\in\N$,
\begin{equation}\label{fuchs}
\frac{\sqrt{2k-1}(k-1)!\big(X_n(k)-(\log n)^k/k!\big)}{(\log
n)^{k-1/2}}~\todistr~ {\rm normal}(0,1).
\end{equation}
Here and hereafter, we write  $\Rightarrow$, ${\overset{{\rm
d}}\longrightarrow}$ and ${\overset{{\rm f.d.d.}}\longrightarrow}$
to denote weak convergence in a functional space, weak convergence
of one-dimensional and finite-dimensional distributions, respectively.
Furthermore, the uniform in $k=1,2,\ldots, o(\log n)$ rate of
convergence in the uniform metric was obtained. On the other hand, in the regime where $\eps \log n < k_n < (e-\eps)\log n$ (with $\eps>0$ fixed) functional limit theorems with non-normal limits  were established in~\cite{chauvin_etal:2001,chauvin_etal:2005,kab_mar_sulz:2017}.

The present
article is a follow-up of \cite{Iksanov+Kabluchko:2018} in which a
functional limit theorem was proved for the random process
$\big(X_{[n^u]}(1),\ldots, X_{[n^u]}(k)\big)_{u\geq 0}$ for each
$k\in\mn$, properly normalized and centered, as $n\to\infty$. In particular, for $u=1$ this result yields  the following multivariate version of~\eqref{fuchs}:
\begin{equation}\label{multivariate}
\Bigg(\frac{(j-1)!\big(X_n(j)-(\log n)^j/j!\big)}{(\log
n)^{j-1/2}}\Bigg)_{j=1,\ldots,k} ~\todistr~ (N_1,\ldots,N_k),
\end{equation}
where $(N_1,\ldots,N_k)$ is a $k$-variate normal random vector with zero mean and covariances
\begin{equation}\label{eq:cov_Ni_Nj}
\E N_i N_j = \frac{1}{i+j-1}, \qquad 1\leq i,j \leq k.
\end{equation}
Let
$(k_n)_{n\in\mn}$ be a sequence of positive numbers satisfying
$k_n\to \infty$ and $k_n=o(\log n)$ as $n\to\infty$. Our purpose
is to investigate weak convergence of the process $(X_n([k_n
u]))_{u>0}$, again properly normalized and centered, thereby
providing information about occupancy of {\it intermediate levels}
in a random recursive tree on $n+1$ vertices. Our main result is given in Theorem \ref{main3}.
\begin{Theorem}\label{main3}
Let $(k_n)_{n\in\mn}$ be a sequence of positive numbers satisfying
$k_n\to \infty$ and $k_n=o(\log n)$ as $n\to\infty$. The following
limit theorem holds for the intermediate levels of a random
recursive tree with $n+1$ vertices:
\begin{equation}\label{clt55}
\Bigg(\frac{[k_n]^{1/2}([k_n u]-1)!\big(X_n([k_n u])-(\log
n)^{[k_n u]}/[k_n u]!\big)}{(\log n)^{[k_n
u]-1/2}}\Bigg)_{u>0}~\tofd~ \Bigg(\int_{[0,\,\infty)}e^{-uy}{\rm
d}B(y)\Bigg)_{u>0},
\end{equation}
where $(B(v))_{v\geq 0}$ is a standard Brownian motion.
\end{Theorem}
\begin{Rem}
The limit process in Theorem \ref{main3} can be defined via
integration by parts $$T(u):=\int_{[0,\infty)}e^{-uy}{\rm
d}B(y)=u\int_0^\infty e^{-uy}B(y){\rm d}y,\quad u>0.$$ The process
$T$ is a.s.\ continuous on $(0,\infty)$. However, it cannot be
defined by continuity at $u=0$ because of the oscillating behavior
of the Brownian motion at $\infty$. This explains that the limit
theorem holds for $u>0$ rather than $u\geq 0$.

It can be checked (details can be found in Section 2 of
\cite{Iksanov:2013}) that $T(u)$ has the same distribution as
$B(1)/\sqrt{2u}$ for each $u>0$. Therefore, we recover
\eqref{fuchs} when taking in \eqref{clt55} $u=1$. Note also
that $$\E T(u)T(v)=(u+v)^{-1},\quad u,v>0.$$ As a consequence of  
$$\E \left( e^uT(e^{2u}) \, e^vT(e^{2v}) \right) = \frac{e^{u+v}}{e^{2u} + e^{2v}} = \frac 1 {2 \cosh (u-v)},\quad u,v\in\mr,$$
a transformed process $\left(e^uT(e^{2u})\right)_{u\in\mr}$ is stationary Gaussian. Finally, observe that on the formal level Theorem~\ref{main3} is consistent with~\eqref{multivariate}. Indeed, taking $i=[k_n u]$ and $j=[k_n v]$ in~\eqref{eq:cov_Ni_Nj} we obtain the covariance $1/([k_nu] + [k_nv] - 1) \sim (u+v)^{-1} k_n^{-1}$ as $n\to\infty$.
\end{Rem}

\section{Connection to a CMJ-branching process}
To prove Theorem \ref{main3} we shall use the same approach as in
\cite{Iksanov+Kabluchko:2018}. The core of this approach is
distributional equality \eqref{basic2} which shows that the
process $(X_n([k_n u]))_{u>0}$ of our interest is naturally
embedded into appropriate Crump-Mode-Jagers branching process
(CMJ-process).

The random recursive tree can be constructed in continuous time as follows. At time $0$, start with a tree consisting of one vertex labeled by $1$ (the root). After an exponential time with unit mean, add to this vertex an offspring labeled by $2$. Any time a new vertex with label $n$ is added to a tree, assign to each vertex of the tree a unit exponential clock that is independent of everything else. Each time some clock rings, add an offspring to the corresponding vertex  and repeat the procedure.  Let us denote by $\tau_n$ the time at which the vertex with label $n+1$ was added to the tree. Then, the tree obtained at time $\tau_n$ has the same probability law as the random recursive tree $\mathcal T_{n+1}$.
Note that $\tau_0=0$ and for each $n\in\N$, the difference  $\tau_n-\tau_{n-1}$ is an exponential random variable with mean $1/n$. Moreover, all such differences are independent.

By construction, the times at which the root of the tree generates offspring form arrival times of a Poisson process with unit intensity. A similar statement holds for any vertex in the tree: if a vertex  was born at time $t$, then the differences between the times at which this vertex generates offspring and $t$ form arrival times of a Poisson process with unit intensity. In the following, we shall generalize this construction by replacing exponential interarrival times with arbitrary positive interarrival times.

Let $(\xi_k)_{k\in\mn}$ be independent copies of a positive random variable $\xi$. Let $S:=(S_n)_{n\in\mn}$ be
the ordinary random walk with jumps $\xi_n$ for $n\in\mn$, that
is, $S_n = \xi_1+\ldots+\xi_n$, $n \in \mn$. The corresponding renewal process $(N(t))_{t\in\R}$ is defined by
$$N(t):=\sum_{k\geq 1}\1_{\{S_k\leq t\}},\quad t\in \R.$$
Let $U(t):=\E N(t)$, where $t\in\mr$, be the renewal function.
For $t \leq  0$, we have $N(t)=0$ a.s.\ and $U(t)=0$.

We are now ready to recall the construction of the Crump-Mode-Jagers branching
process relevant to us. We are only interested in the special case when the CMJ-process is generated by the random
walk $S$. At time $\tau_0=0$ there is one individual, called the
ancestor. The ancestor produces offspring (the first generation)
with birth times given by a point process $\mm = \sum_{n\geq 1}
\delta_{S_n}$ on $\R_+:=[0,\infty)$. The first generation produces
the second generation. The shifts of birth times of the second
generation individuals with respect to their mothers' birth times
are distributed according to independent copies of the same point
process $\mm$. The second generation produces the third one, and
so on. All individuals act independently of each other.

For $k\in\mn$, denote by $Y_k(t)$ the number of the $k$th
generation individuals with birth times $\leq t$. For example,
$Y_1(t)=N(t)$ for $t\geq 0$.  For $n\in\mn$,
denote by $\tau_n$ the birth time of the $n$th individual (in the
chronological order of birth times, excluding the ancestor).

Now we are ready to state the basic observation for the proof
of Theorem \ref{main3}. In the special case when $\xi$ has exponential distribution with unit mean, the individuals of the CMJ-process correspond to vertices of the random recursive tree, the ancestor at time $0$ corresponds to the root, and the generation of the individual corresponds to the distance to the root. It follows that for a sequence $(k_n)$ with
$\lim_{n\to\infty}k_n=\infty$ we have
\begin{equation}\label{basic2}
(X_n([k_n u]))_{u>0}~\overset{{\rm d}}{=}~ (Y_{[k_n
u]}(\tau_n))_{u>0},\quad n\in\N.
\end{equation}
The basic decomposition we need reads
$$
Y_k(t)=\sum_{i\geq 1}Y^{(i)}_{k-1}(t-S_i),\quad t\geq 0, k\geq 2,
$$
where $Y_j^{(i)}(t)$ is the number of successors in the $(j+1)$st generation of the $1$st generation individual born at time $S_i$ that are born in the interval $[S_i,t+S_i]$.  By the definition of the CMJ-process, $(Y_j^{(1)}(t))_{t\geq
0}$, $(Y_j^{(2)}(t))_{t\geq 0},\ldots$ are independent copies of
$(Y_j(t))_{t\geq 0}$ which are independent of $S$.  Note
that, for $k\geq 2$, $(Y_k(t))_{t\geq 0}$ is a particular instance
of a random process with immigration at the epochs of a renewal
process which is a renewal shot noise process with random and
independent response functions (the term was introduced in
\cite{Iksanov+Marynych+Meiners:2017}; see also \cite{Iksanov:2017}
for a review).

For $t\geq 0$ and $k\in\mn$, we define $U_k(t):=\E Y_k(t)$. Then,  $U_1(t)=U(t)$ and
$$U_k(t)=\int_{[0,\,t]}U_{k-1}(t-y){\rm d}U(y),\quad k\geq 2,~ t\geq 0.$$
In the special case when the distribution of $\xi$ is exponential with unit mean, we have
$U_1(t)=t$ for $t\geq 0$ and, more generally,
\begin{equation}\label{formula}
U_k(t)=\frac{t^k}{k!},\quad k\in\mn,~ t\geq 0,
\end{equation}
which follows from the recursive formula
$U_k(t)=\int_0^t U_{k-1}(y){\rm d}y$ for $k\geq 2$.

Theorem \ref{main3} will be obtained as a consequence of the
following two results.
\begin{Theorem}\label{main4}
Let $k(t)$ be any positive function satisfying $k(t)\to \infty$
and $k(t)=o(t)$ as $t\to\infty$. Assume that the distribution of
$\xi$ is exponential with unit mean. Then
$$\bigg(\frac{[k(t)]^{1/2}([k(t)u]-1)!}{t^{[k(t)u]-1/2}}\sum_{j\geq
1}\big(Y^{(j)}_{[k(t)u]-1}(t-S_j)-U_{[k(t)u]-1}(t-S_j)\big)\1_{\{S_j\leq
t\}}\bigg)_{u>0}~\tofdt~ 0,$$ where, recalling \eqref{formula},
$U_{[k(t)u]-1}(t)=t^{[k(t)u]-1}/([k(t)u]-1)!$ for $u>0$ and $t>0$.
\end{Theorem}

In what follows we denote by $D(0,\infty)$ ($D[0,\infty)$) the Skorokhod space of
right-continuous functions defined on $(0,\infty)$ (on $[0,\infty)$) with finite
limits from the left at positive points.
\begin{Theorem}\label{main5}
Let $k(t)$ be any positive function satisfying $k(t)\to \infty$
and $k(t)=o(t)$ as $t\to\infty$. Assume that $\sigma^2:={\rm
Var}\,\xi\in (0,\infty)$ (the distribution of $\xi$ is not assumed
exponential). Then 
\begin{equation}\label{clt22}
\left(\frac{[k(t)]^{1/2}([k(t)u]-1)!}{\sqrt{\sigma^2\mu^{-2[k(t)u]-1}t^{2[k(t)u]-1}}}\bigg(\sum_{j\geq
1}\frac{(t-S_j)^{[k(t)u]-1} \1_{\{S_j\leq
t\}}}{([k(t)u]-1)!\mu^{[k(t)u]-1}}-\frac{t^{[k(t)u]}}{([k(t)u])!\mu^{[k(t)u]}}\bigg)\right)_{u>0}
~\toweakt~ (T(u))_{u>0}
\end{equation}
in the $J_1$-topology on $D(0,\infty)$, where $\mu=\E\xi<\infty$.
In particular, we have in \eqref{clt22} weak convergence of the
finite-dimensional distributions.
\end{Theorem}

\section{Proof of Theorem \ref{main3}}

Throughout the proof we assume that $\xi$ is exponentially
distributed with unit mean. In particular, we have $\mu=\sigma^2=1$
in the notation of Theorem \ref{main5}. Keeping this in mind, a
combination of Theorems \ref{main4} and \ref{main5} yields
\begin{equation}\label{clt223}
\bigg(\frac{[k(t)]^{1/2}([k(t)u]-1)!}{t^{[k(t)u]-1/2}}\bigg(Y_{[k(t)u]}(t)-\frac{t^{[k(t)u]}}{([k(t)u])!}\bigg)\bigg)_{u>0}
~\tofdt~ (T(u))_{u>0}
\end{equation}
for any positive function $k(t)$ satisfying $k(t)\to\infty$ and
$k(t)=o(t)$ as $t\to\infty$.

Given a sequence $(k_n)$ as in the statement of Theorem \ref{main3}, define the function $k=k(t) = k_{[e^t]}$. Note that $k(\log n) = k_n$, $k(t)\to\infty$ and $k(t)=o(t)$ as $t\to\infty$. Choose any positive $\ell=\ell(t)$ satisfying $\lim_{t\to\infty}(\ell^2(t)/(t
k(t)))=\infty$ and $\ell(t)=o(t)$ as $t\to\infty$. For instance, one can take $\ell(t)=t^{3/4}k^{1/4}(t)$. For $n\in\mn$, set $a_n=\log
n/\ell(\log n)$.

Recall from the previous section that $\tau_n$ is the sum of $n$ independent
exponentially distributed random variables with means $1,1/2,\ldots,1/n$.
This implies that $\big(\tau_n-(1+1/2+\ldots+1/n)\big)_{n\in\mn}$ is a square
integrable (hence, convergent) martingale with respect to the
natural filtration. As a consequence, $\tau_n-\log n$ converges
a.s., whence $\lim_{n\to\infty}(\tau_n-\log n)/a_n=0$ a.s.
Equivalently, given $\varepsilon>0$ there exists an a.s.\ finite
$N>0$ such that $\log n-\varepsilon a_n\leq \tau_n\leq \log
n+\varepsilon a_n$ whenever $n\geq N$. In what follows, for ease
of notation we write $1$ for $\varepsilon$.

Fix any positive and finite $T_1<T_2$. For $u>0$ and $n\in\mn$,
set $$K_n(u):=\frac{[k(\log n)]^{1/2}([k(\log n)u]-1)!}{(\log
n)^{[k(\log n)u]-1/2}}\bigg(Y_{[k(\log n)u]}(\tau_n)-\frac{(\log
n)^{[k(\log n)u]}}{([k(\log n)u])!}\bigg).$$ In view of \eqref{basic2}, it suffices to show that
$$(K_n(u))_{u>0} ~\tofd~ (T(u))_{u>0}.$$
Obviously, for all
$\delta>0$,
$$\lim_{n\to\infty}\Prob\{\sup_{u\in [T_1,\, T_2]}K_n(u)\1_{\{N>n\}}
>\delta\}=0.$$ Since $Y_{[k(t)u]}(t_1)\leq Y_{[k(t)u]}(t_2)$ whenever $t_1<t_2$
we infer
\begin{eqnarray*}
K_n(u)\1_{\{n\geq N\}}&\leq& \frac{[k(\log n)]^{1/2}([k(\log
n)u]-1)!}{(\log n+a_n)^{[k(\log n)u]-1/2}}\bigg(Y_{[k(\log
n)u]}(\log n+a_n)-\frac{(\log n+a_n)^{[k(\log n)u]}}{([k(\log
n)u])!}\bigg)\\&\times&\Big(1+\frac{1}{\ell(\log n)}\Big)^{[k(\log
n)u]-1/2}\1_{\{n\geq N\}}\\&+& \frac{(k(\log n)\log
n)^{1/2}}{[k(\log n)u]}\Bigg(\Big(1+\frac{1}{\ell(\log
n)}\Big)^{[k(\log n)u]}-1\Bigg)=: I_n(u)\times \eta_n(u)+J_n(u)
\end{eqnarray*}

Putting in \eqref{clt223} $t=\log n$ when $t$ is an argument of
the function $k$ and $t=\log n+a_n$, otherwise we infer
$(I_n(u))_{u>0}\tofd (T(u))_{u>0}$. Further, our choice of $\ell$
entails $\lim_{n\to\infty}(\ell(\log n)/k(\log n))=\infty$, whence
$$\lim_{n\to\infty}\sup_{u\in
[0,\,T_2]}|\eta_n(u)-1|=0\quad\text{a.s.}$$ Finally, for large
enough $n$,
$$\sup_{T_1\leq u\leq T_2} J_n(u)\leq \frac{2(k(\log n)\log n)^{1/2}}{\ell(\log
n)}.$$ The right-hand side converges to zero as $n\to\infty$ by
our choice of $\ell$. Combining pieces together we obtain
$(I_n(u)\times \eta_n(u)+J_n(u))_{u>0}\tofd (T(u))_{u>0}$. The
same conclusion for the lower bound of $K_n(u)\1_{\{n\geq N\}}$
can be derived similarly. The proof of Theorem \ref{main3} is
complete.

\section{Proof of Theorem \ref{main4}}

We first prove the following.
\begin{Lemma}\label{mom_asy}
Assume that the distribution of $\xi$ is exponential with unit mean
and let $k=k(t)\to\infty$ through integers and $k(t)=o(t)$ as
$t\to\infty$. Then
$$\E \Big(\sum_{j\geq
1}\big(Y^{(j)}_{k-1}(t-S_j)-U_{k-1}(t-S_j)\big)\1_{\{S_j\leq
t\}}\Big)^2~\sim~
\frac{1}{4}\frac{t^{2k}}{(k!)^2}\Big(\frac{k}{t}\Big)^2,\quad
t\to\infty.$$
\end{Lemma}
\begin{proof}
Without any restrictions on the distribution of a positive random
variable $\xi$ the following formulas were obtained in Lemma 4.2
of \cite{Iksanov+Kabluchko:2018}: for $k\geq 2$ and $t\geq 0$
\begin{eqnarray}\label{aux5}
D_k(t)&:=&{\rm Var}\,Y_k(t)=\E \bigg(\sum_{j\geq
1}\big(Y^{(j)}_{k-1}(t-S_j)-U_{k-1}(t-S_j)\big)\1_{\{S_j\leq
t\}}\bigg)^2\\&+& \E\bigg(\sum_{j\geq
1}U_{k-1}(t-S_j)\1_{\{S_j\leq t\}}-U_k(t)\bigg)^2;\notag
\end{eqnarray}
\begin{equation}\label{aux10}
\E \bigg(\sum_{j\geq
1}\big(Y^{(j)}_{k-1}(t-S_j)-U_{k-1}(t-S_j)\big)\1_{\{S_j\leq
t\}}\bigg)^2=\int_{[0,\,t]}D_{k-1}(t-y){\rm d}U(y);
\end{equation}
\begin{eqnarray}\label{mom}
\E\bigg(\sum_{j\geq 1}U_{k-1}(t-S_j)\1_{\{S_j\leq
t\}}-U_k(t)\bigg)^2&=&2\int_{[0,\,t]} U_{k-1}(t-y)U_k(t-y){\rm
d}U(y)\\&+&\int_{[0,\,t]}U^2_{k-1}(t-y){\rm d}U(y)-U_k^2(t).\notag
\end{eqnarray}

Assume now that the distribution of $\xi$ is exponential with unit
mean. Invoking \eqref{formula} and \eqref{mom} we obtain
\begin{eqnarray*}
\E\bigg(\sum_{j\geq 1}U_{k-1}(t-S_j)\1_{\{S_j\leq
t\}}-U_k(t)\bigg)^2&=&2\int_0^t U_{k-1}(y)U_k(y){\rm d}y+\int_0^t
U^2_{k-1}(y){\rm
d}y-U^2_k(t)\\&=&\frac{t^{2k-1}}{((k-1)!)^2(2k-1)},\quad k\geq 2,
t\geq 0.
\end{eqnarray*}
Using the latter formula together with \eqref{aux5} and
\eqref{aux10} we have
$$D_k(t)=\int_0^t D_{k-1}(y){\rm d}y+\E\bigg(\sum_{j\geq 1}U_{k-1}(t-S_j)\1_{\{S_j\leq
t\}}-U_k(t)\bigg)^2=\int_0^t D_{k-1}(y){\rm
d}y+\frac{t^{2k-1}}{((k-1)!)^2(2k-1)}.$$ This in combination with
the boundary condition $D_1(t)=t$ immediately gives
$$D_k(t)=\sum_{i=0}^{k-1}\frac{t^{k+i}}{(i!)^2}\frac{(2i)!}{(k+i)!},\quad k\in\mn, t\geq 0$$
whence, recalling \eqref{aux10}, $$\E \bigg(\sum_{j\geq
1}\big(Y^{(j)}_{k-1}(t-S_j)-U_{k-1}(t-S_j)\big)\1_{\{S_j\leq
t\}}\bigg)^2 =\int_0^t D_{k-1}(y){\rm
d}y=\sum_{i=0}^{k-2}\frac{t^{k+i}}{(i!)^2}\frac{(2i)!}{(k+i)!},\quad
k\geq 2, t\geq 0.$$ We claim that the left-hand side is asymptotic
to the $(k-2)$nd term of the last sum which is
$$\frac{t^{2k-2}}{((k-2)!)^2}\frac{(2k-4)!}{(2k-2)!}~\sim~\frac{1}{4}\frac{t^{2k}}{(k!)^2}\Big(\frac{k}{t}\Big)^2,\quad
t\to\infty.$$ To prove this, it suffices to show that
$$\lim_{t\to\infty} \sum_{i=1}^{k-3}\frac{A(i,k,t)}{t^{k-i-2}}=0,$$
where
$$A(i,k,t):=\frac{(k!)^2 (2i)!}{(i!)^2 (k+i)! k^2}.$$ Using the inequality
\begin{equation*}
(2\pi n)^{1/2}(ne^{-1})^n \leq n!\leq e(2\pi n)^{1/2}(ne^{-1})^n,
\quad n\in\mn
\end{equation*}
which is a consequence of the Stirling formula in the form
\begin{equation*}
n!= (2\pi n)^{1/2}(ne^{-1})^n e^{\theta_n/(12 n)}, n\in\mn,
\end{equation*}
where $\theta_n\in (0,1)$, we obtain
\begin{equation}\label{aux7}
\frac{1}{2^{1/2}e} A(i,k,t)\leq
\frac{4^i}{i^{1/2}}\frac{k^{2k-1}}{(k+i)^{k+i+1/2}e^{k-i-2}}\leq
4^i k^{1/2}\Big(\frac{k}{e}\Big)^{k-i-2}.
\end{equation}
This yields
$$\frac{1}{2^{1/2}e}\sum_{i=1}^{[k/2]-1}\frac{A(i,k,t)}{t^{k-i-2}}\leq
k^{1/2}\sum_{i=k-[k/2]-1}^{k-3}\Big(\frac{4k}{et}\Big)^i\leq
k^{1/2}\Big(\frac{4k}{et}\Big)^{k-[k/2]-1}\Big(1-\frac{4k}{et}\Big)^{-1}$$
having utilized $4^i\leq 4^{k-i-2}$ which holds for $1\leq i\leq
[k/2]-1$. The right-hand side goes to zero as $t\to\infty$.
Another appeal to \eqref{aux7} gives
$$\frac{1}{2^{1/2}e}\sum_{i=[k/2]}^{k-3}\frac{A(i,k,t)}{t^{k-i-2}}\leq k^{1/2}
\Big(\frac{k}{et}\Big)^{k-[k/2]-2}\sum_{i=[k/2]}^{k-3}4^i\leq
\frac{1}{3} k^{1/2}4^{k-2}\Big(\frac{k}{et}\Big)^{k-[k/2]-2}.$$
The right-hand side converges to zero as $t\to\infty$ which
completes the proof of the lemma.
\end{proof}

We are ready to prove Theorem \ref{main4}.
\begin{proof}[Proof of Theorem \ref{main4}]
In view of the Cram\'{e}r-Wold device and Markov's inequality weak
convergence of the finite-dimensional distributions to the zero
vector is a consequence of
\begin{eqnarray*}
&&\frac{(([ku]-1)!)^2 k}{t^{2[ku]-1}} \E \Big(\sum_{j\geq
1}\big(Y^{(j)}_{[ku]-1}(t-S_j)-U_{[ku]-1}(t-S_j)\big)\1_{\{S_j\leq
t\}}\Big)^2\\&\sim~&\frac{(([ku]-1)!)^2 k}{t^{2[ku]-1}}
\frac{1}{4}
\frac{t^{2[ku]}}{(([ku])!)^2}\Big(\frac{[ku]}{t}\Big)^2~\sim~
\frac{k}{4t} ~\to~ 0
\end{eqnarray*}
for each $u>0$ as $t\to\infty$. Here, we have used Lemma
\ref{mom_asy} for the first asymptotic equivalence.
\end{proof}

\section{Proof of Theorem \ref{main5}}
First, we use the Cram\'{e}r-Wold device to prove weak convergence
of finite-dimensional distributions in \eqref{clt22}, that is, for
any $j\in\mn$, any real $\alpha_1,\ldots, \alpha_j$ and any
$0<u_1<\ldots<u_j<\infty$
\begin{equation}\label{fidi}
\sum_{i=1}^j \alpha_i\frac{k^{1/2}([ku_i]-1)! Z(ku_i,
t)}{\sqrt{\sigma^2\mu^{-2[ku_i]-1}t^{2[ku_i]-1}}}~\todistrt~
\sum_{i=1}^j \alpha_i u_i\int_0^\infty B(y)e^{-u_i y}{\rm d}y,
\end{equation}
where $$Z(ku, t):=\sum_{j\geq
1}\Big(\frac{(t-S_j)^{[ku]-1}}{([ku]-1)!\mu^{[ku]-1}}\1_{\{S_j\leq
t\}}-\frac{t^{[ku]}}{([ku])!\mu^{[ku)]}}\Big).$$ To ease notation,
here and hereafter, we write $k$ for $k(t)$.

We have for any $u, T>0$ and sufficiently large $t$
\begin{eqnarray*}
&&\frac{k^{1/2}([ku]-1)! Z(ku,
t)}{\sqrt{\sigma^2\mu^{-2[ku]-1}t^{2[ku]-1}}}\\&=&\frac{k^{1/2}}{\sqrt{\sigma^2\mu^{-3}t^{2[ku]-1}}}\int_{[0,\,t]}(t-y)^{[ku]-1}{\rm
d}(N(y)-\mu^{-1}y)\\&=&\frac{k^{1/2}([ku]-1)}{\sqrt{\sigma^2\mu^{-3}t^{2[ku]-1}}}\bigg(\int_0^{Tk/t}(N(y)-\mu^{-1}y)(t-y)^{[ku]-2}{\rm
d}y+\int_{Tk/t}^t (N(y)-\mu^{-1}y)(t-y)^{[ku]-2}{\rm
d}y\bigg)\\&=& \frac{[ku]-1}{k}\int_0^T
\frac{N((t/k)y)-\mu^{-1}(t/k)y}{\sqrt{\sigma^2\mu^{-3}t/k}}\big(1-\frac{y}{k}\big)^{[ku]-2}{\rm
d}y\\&+&
\frac{k^{1/2}([ku]-1)}{\sqrt{\sigma^2\mu^{-3}t^{2[ku]-1}}}\int_{Tk/t}^t
(N(y)-\mu^{-1}y)(t-y)^{[ku]-2}{\rm d}y.
\end{eqnarray*}
By Theorem 3.1 on p.~162 in \cite{Gut:2009}
$$\frac{N(t\cdot)-\mu^{-1}(\cdot)}{\sqrt{\sigma^2 \mu^{-3}t}}~\toweakt~ B(\cdot)$$
in the $J_1$-topology on $D[0,\infty)$. By Skorokhod's representation
theorem there exist versions $\widehat{N}$ and $\widehat{B}$ such
that
\begin{equation}\label{cs}
\lim_{t\to\infty}\sup_{0\leq y\leq
T}\bigg|\frac{\widehat{N}(ty)-\mu^{-1}ty}{\sqrt{\sigma^2\mu^{-3}t}}-\widehat{B}(y)\bigg|=0\quad\text{a.s.}
\end{equation}
for all $T>0$. Using \eqref{cs} with $t/k$ replacing $t$ in
combination with
$$\lim_{t\to\infty}\sup_{0\leq y\leq T}\Big|\big(1-\frac{y}{k(t)}\big)^{[k(t)u]-2}-e^{-uy}\Big|=0$$ we infer
$$\lim_{t\to\infty}\sup_{u\in [0,T]}\Big|\frac{\widehat{N}((t/k)y)-\mu^{-1}(t/k)y}{\sqrt{\sigma^2\mu^{-3}t/k}}\big(1-\frac{y}{k}\big)^{[ku]-2}
-\widehat{B}(y)e^{-uy}\Big|=0\quad \text{a.s.}$$ This shows that
$$\lim_{t\to\infty}\sum_{i=1}^j \alpha_i\frac{[ku_i]-1}{k} \int_0^T \frac{\widehat{N}((t/k)y)-\mu^{-1}(t/k)y}{\sqrt{\sigma^2\mu^{-3}t/k}}
\big(1-\frac{y}{k}\big)^{[ku_i]-2}{\rm d}y=\sum_{i=1}^j \alpha_i
u_i\int_0^T \widehat{B}(y)e^{-u_iy}{\rm d}y\quad \text{a.s.}$$ and
thereupon
$$\sum_{i=1}^j \alpha_i \frac{[ku_i]-1}{k}\int_0^T
\frac{N((t/k)y)-\mu^{-1}(t/k)y}{\sqrt{\sigma^2\mu^{-3}t/k}}\big(1-\frac{y}{k}\big)^{[ku_i]-2}{\rm
d}y~\todistrt~\sum_{i=1}^j \alpha_i u_i\int_0^T B(y)e^{-u_iy}{\rm
d}y.$$ Since $\lim_{T\to\infty}\sum_{i=1}^j \alpha_i u_i \int_0^T
B(y)e^{-u_iy}{\rm d}y=\sum_{i=1}^j\alpha_i u_i \int_0^\infty
B(y)e^{-u_iy}{\rm d}y$ a.s.\ it remains to prove that
$$\lim_{T\to\infty}{\lim\sup}_{t\to\infty}\,\Prob\bigg\{\bigg|\sum_{i=1}^j\alpha_i \frac{k^{1/2}([ku_i]-1)}{t^{[ku_i]-1/2}}\int_{Tk/t}^t
(N(y)-\mu^{-1}y)(t-y)^{[ku_i]-2}{\rm
d}y\bigg|>\varepsilon\bigg\}=0$$ for all $\varepsilon>0$. In view
of Markov's inequality and the fact that $\E |N(y)-\mu^{-1}y|\sim
\sigma\mu^{-3/2} \E |B(1)|y^{1/2}$ as $y\to\infty$ (see Theorem
8.4 on p.~98 in \cite{Gut:2009}) the latter is a consequence of
$$\lim_{T\to\infty}{\lim\sup}_{t\to\infty}\,\frac{k^{1/2}([ku]-1)}{t^{[ku]-1/2}}\int_{Tk/t}^t
y^{1/2}(t-y)^{[ku]-2}{\rm d}y=0$$ for $u>0$. To justify it,
observe that
$$\frac{k^{1/2}([ku]-1)}{t^{[ku]-1/2}}\int_{Tk/t}^t
y^{1/2}(t-y)^{[ku]-2}{\rm d}y=\frac{[ku]-1}{k}\int_T^k
y^{1/2}\big(1-\frac{y}{k}\big)^{[ku]-2}{\rm d}y~\to~
u\int_T^\infty y^{1/2}e^{-uy}{\rm d}y$$ as $t\to\infty$ by
Lebesgue's dominated convergence theorem. The proof of
\eqref{fidi} is complete. For later use, we note that exactly the
same argument leads to
\begin{equation}\label{intermediate}
\int_0^k
\bigg|\frac{N((t/k)y)-\mu^{-1}(t/k)y}{\sqrt{\sigma^2\mu^{-3}t/k}}\bigg|
\big(1-\frac{y}{k}\big)^{[ku]-2}(1+y){\rm
d}y~\todistrt~\int_0^\infty |B(y)|e^{-uy}(1+y){\rm d}y
\end{equation}
for $u>0$.

It remains to prove tightness in \eqref{clt22}. By Theorem 15.5 in
\cite{Billingsley:1968} it suffices to show that for any
$0<a<b<\infty$, $\varepsilon>0$ and $\gamma\in (0,1)$ there exist
$t_0>0$ and $\delta>0$ such that
\begin{equation}\label{tight} \Prob\bigg\{\sup_{a\leq u,v\leq
b, |u-v|\leq \delta}\,\bigg|\frac{k^{1/2}([ku]-1)! Z(ku,
t)}{\sqrt{\sigma^2\mu^{-2[ku]-1}t^{2[ku]-1}}}-\frac{k^{1/2}([kv]-1)!
Z(kv,t)}{\sqrt{\sigma^2\mu^{-2[kv]-1}t^{2[kv]-1}}}\bigg|>\varepsilon\bigg\}\leq
\gamma
\end{equation}
for all $t\geq t_0$. As a preparation for the proof of
\eqref{tight}, let us note that for $a\leq u,v\leq b$ such that
$|u-v|\leq \delta$, $y\in [0,k]$ and large enough $k$ we have
\begin{eqnarray*}
&&\bigg|\frac{[ku]-1}{k}\big(1-\frac{y}{k}\big)^{[ku]-2}-\frac{[kv]-1}{k}\big(1-\frac{y}{k}\big)^{[kv]-2}\bigg|\\&=&\big(1-\frac{y}{k}\big)^{[k(u\wedge
v)]-2} \bigg|\frac{[k(u\vee v)]-[k(u\wedge
v)]}{k}\big(1-\frac{y}{k}\big)^{[k(u\vee v)]-[k(u\wedge
v)]}\\&-&\frac{[k(u\wedge
v)]-1}{k}\big(1-\big(1-\frac{y}{k}\big)^{[k(u\vee v)]-[k(u\wedge
v)]}\big)\bigg|\\&\leq&
\big(1-\frac{y}{k}\big)^{[ka]-2}\bigg(\frac{[k(u\vee
v)]-[k(u\wedge v)]}{k}+b\frac{[k(u\vee v)]-[k(u\wedge
v)]}{k}y\bigg)\\&\leq&
C|u-v|\big(1-\frac{y}{k}\big)^{[ka]-2}(1+y)\leq C\delta
\big(1-\frac{y}{k}\big)^{[ka]-2}(1+y)
\end{eqnarray*}
for appropriate constant $C>0$. With this at hand
\begin{eqnarray*}
&&\sup_{a\leq u,v\leq b, |u-v|\leq
\delta}\,\bigg|\frac{k^{1/2}([ku]-1)! Z(ku,
t)}{\sqrt{\sigma^2\mu^{-2[ku]-1}t^{2[ku]-1}}}-\frac{k^{1/2}([kv]-1)!
Z(kv,t)}{\sqrt{\sigma^2\mu^{-2[kv]-1}t^{2[kv]-1}}}\bigg|\\&=&
\sup_{a\leq u,v\leq b, |u-v|\leq \delta}\,\bigg|\int_0^k
\frac{N((t/k)y)-\mu^{-1}(t/k)y}{\sqrt{\sigma^2\mu^{-3}t/k}}
\Big(\frac{[ku]-1}{k}\big(1-\frac{y}{k}\big)^{[ku]-2}-
\frac{[kv]-1}{k}\big(1-\frac{y}{k}\big)^{[kv]-2}\Big){\rm
d}y\bigg|\\&\leq& C\delta \int_0^k
\bigg|\frac{N((t/k)y)-\mu^{-1}(t/k)y}{\sqrt{\sigma^2\mu^{-3}t/k}}\bigg|
\big(1-\frac{y}{k}\big)^{[ka]-2}(1+y){\rm d}y.
\end{eqnarray*}
Recalling \eqref{intermediate} and choosing $\delta$ sufficiently
small we arrive at \eqref{tight}. The proof of Theorem \ref{main5}
is complete.

\section{Open problem}

It is an interesting open problem whether weak convergence of the
finite-dimensional distributions in Theorem \ref{main3} can be
strengthened to weak convergence on $D(0,\infty)$. To ensure this
it is sufficient to show that the left-hand side of the centered
formula in Theorem \ref{main4} converges weakly to the zero
function on $D(0,\infty)$. Indeed, if the latter were true, the
proof of Theorem \ref{main3} would only require an inessential
modification. However, we have been able neither prove, nor
disprove the aforementioned functional version of Theorem
\ref{main4}.

\vspace{1cm} \noindent   {\bf Acknowledgements}  \quad
\footnotesize %The authors thank the anonymous referee for several
%valuable comments that helped improving the presentation.
A part of this work was done while A.~Iksanov was visiting
M\"{u}nster in January 2018. A.I. gratefully acknowledges
hospitality and the financial support by DFG SFB 878 ``Geometry,
Groups and Actions''.

\end{document}